\documentclass{amsart}

\usepackage{xypic}
\usepackage{amsrefs}
\usepackage{longtable,booktabs}

\newtheorem{thm}{Theorem}[section]
\newtheorem{prop}[thm]{Proposition}
\newtheorem{cor}[thm]{Corollary}
\newtheorem{lemma}[thm]{Lemma}

\theoremstyle{definition}
\newtheorem{defn}[thm]{Definition}
\newtheorem{remark}[thm]{Remark}

\newcommand{\C}{\mathbb{C}}
\newcommand{\F}{\mathbb{F}}
\newcommand{\M}{\mathbb{M}}

\newcommand{\map}{\rightarrow}
\newcommand{\cl}{\mathrm{cl}}
\newcommand{\tmf}{\mathit{tmf}}
\newcommand{\mmf}{\mathit{mmf}}

\DeclareMathOperator{\Ext}{Ext}	
\DeclareMathOperator{\Sq}{Sq}

\begin{document}

\title{The Mahowald operator in the cohomology of the 
Steenrod algebra}

\author{Daniel C.\ Isaksen}
\address{Department of Mathematics\\
Wayne State University\\
Detroit, MI 48202, USA}
\email{isaksen@wayne.edu}
\thanks{The author was supported by NSF grant DMS-1606290.}



\date{\today}

\begin{abstract}
We study the Mahowald operator $M = \langle g_2, h_0^3, - \rangle$
in the cohomology of the Steenrod algebra.  We show that the
operator interacts well with the cohomology of $A(2)$, in both the
classical and $\C$-motivic contexts.
This generalizes previous work of Margolis, Priddy, and Tangora.
\end{abstract}

\maketitle

\section{Introduction}
\label{sctn:intro}

The cohomology of the Steenrod algebra is an algebraic object
that serves as the input to the Adams spectral sequence.
Therefore, its computation is of fundamental importance to the
study of the stable homotopy groups of spheres.
The goal of this note is to study
part of the cohomology of the Steenrod algebra that
displays some regular structure.  We work in the $\C$-motivic
context.  Our results have immediate classical
consequences, most of which are already known \cite{MPT70} or can be
readily deduced from the results of \cite{MPT70}.
The ultimate goal of this
study is to serve as an aid in a detailed analysis of the Adams spectral
sequence \cite{IWX}.

Let $A$ be the $\C$-motivic 
Steenrod algebra at the prime $2$
\cite{Voevodsky03a} \cite{Isaksen14}.  Let $\M_2 = \F_2[\tau]$ be the 
$\C$-motivic cohomology of a point 
with $\F_2$-coefficients \cite{Voevodsky03b}.  We are interested in the
algebraic object
$\Ext_\C = \Ext_{A} (\M_2, \M_2)$ because 
it serves as the $E_2$-page for the $\C$-motivic
Adams spectral sequence.
These $\Ext$ groups are of increasingly wild complexity as the
dimension increases.  The May spectral sequence \cite{May64} can be used
to compute them in a range.  Machines can compute in an even
larger range 
\cite{Bruner89} \cite{Bruner93} \cite{Bruner97} \cite{Nassau}. 
In either case,
these methods cannot determine the entire structure
because it is of infinite complexity.

Nevertheless, parts of the computation display regularity.
For example, Adams described a regular $v_1$-periodic
pattern near the ``top of the Adams chart", i.e., when the
Adams filtration is large relative to the stem \cite{Adams66b}.  
May 
extended this $v_1$-periodicity to a larger range \cite{May-unpublished}.
(See also \cite{Li19} for results about $\C$-motivic $v_1$-periodicity.)

The goal of this article is similar.  We study the
Mahowald operator $\langle g_2, h_0^3, - \rangle$, which is 
defined on all elements $x$ such that $h_0^3 x$.  We will 
show that the Mahowald operator behaves regularly in a certain way.

This article is very much inspired by the work of
Margolis, Priddy, and Tangora \cite{MPT70}.
We are extending those results in two senses.
First, we are working in the $\C$-motivic, rather than
classical, context.  Classical results can easily be
deduced from our $\C$-motivic results by inverting $\tau$.
Second, we work with a larger subalgebra of the Steenrod algebra,
and therefore can detect more classes.

When discussing $\Ext$ groups, we grade elements in the
form $(s,f,w)$, where $s$ is the stem, $f$ is the Adams filtration,
and $w$ is the motivic weight.
See \cite{Isaksen14} \cite{Isaksen14a} for more details
about notation and for specific computations.

Recall that $A(2)$ is the $\M_2$-subalgebra of $A$
generated by $\Sq^1$, $\Sq^2$, and $\Sq^4$.
We have a complete understanding of its cohomology,
i.e., of 
$\Ext_{A(2)} = \Ext_{A(2)}(\M_2, \M_2)$ \cite{Isaksen09}.

\begin{thm}
\label{thm:main}
There exists a sub-Hopf algebra $B$ of the $\C$-motivic Steenrod algebra
(defined below in Definition \ref{defn:B}) such that:
\begin{enumerate}
\item
$\Ext_B = \Ext_B(\M_2, \M_2)$ is isomorphic to 
$\M_2[v_3] \underset{\M_2}{\otimes} \Ext_{A(2)}$,
where $v_3$ has degree $(14,1,7)$.
\item
the map
$p_*: \Ext_{\C} \map \Ext_B$
takes $M x$ to the product $(e_0 v_3^2 + h_1^3 v_3^3) p_*(x)$, 
whenever $M x$ is defined.
Also, $p_*$ takes the indeterminacy in $M x$ to zero.
\end{enumerate}
\end{thm}

The theorem can also be stated in the classical context,
in an essentially identical form.

Theorem \ref{thm:main}
allows us to extract much information about the global structure
of $\Ext_\C$.  The following corollary 
gives partial information about the structure
of $\Ext_\C$ in very high dimensions.

\begin{cor}
\label{cor:main}
Let $x$ be an element of $\Ext_\C$ such that
$h_1^3 x = 0$.
Also suppose that the image of $e_0^k x$ 
in $\Ext_{A(2)}$ is non-zero for some $k \geq 0$.
Then 
$M^k x$ is non-zero in $\Ext_\C$.
\end{cor}

\begin{proof}
By Theorem \ref{thm:main},
$p_*(M^k x)$ is non-zero in $\Ext_B$.
Therefore, $M^k x$ must be non-zero.
\end{proof}

For example, $x$ could be $h_1$, $h_2$, $P^i h_1$, $c_0$, or many
other possibilities as well.
Some, but not all, cases of Corollary \ref{cor:main} are already covered
by \cite{MPT70}.

Here is an even more explicit
illustration of the kind of information that can be deduced
from Corollary \ref{cor:main}.
Consider the element $h_0 d_0$.  We have that
$h_0 d_0 e_0^k$ is non-zero in $\Ext_{A(2)}$ for all
$k \geq 0$. (However, $\tau^2 h_0 d_0 e_0^k = 0$ in $\Ext_{A(2)}$
when $k \geq 2$).  We can conclude that $M^k h_0 d_0$
is non-zero in $\Ext_\C$ for all $k \geq 0$, even though
it may be annihilated by powers of $\tau$.

Theorem \ref{thm:main} detects additional phenomena in $\Ext_A$
that are not captured by Corollary \ref{cor:main}.
Namely, the theorem can be used to study classes in 
$\Ext_A$ whose image under $p_*$ is non-zero.
Some examples are listed in Table \ref{tab:p*}.  This list is
far from exhaustive.

\begin{longtable}{lll}
\caption{Some values of the map $p_*:\Ext_\C \map \Ext_B$
\label{tab:p*} 
} \\
\toprule
$(s, f, w)$ & $x$ & $p_*(x)$ \\
\midrule \endfirsthead
\caption[]{Some values of the map $p_*:\Ext_\C \map \Ext_B$} \\
\toprule
$(s, f, w)$ & $x$ & $p_*(x)$ \\
\midrule \endhead
\bottomrule \endfoot
$(53,10,28)$ & $M P$ & $P e_0 v_3^2 + P h_1^3 v_3^3$ \\
$(56,10,29)$ & $\Delta^2 h_1 h_3$ & $\tau P g v_3^2$ \\ 
$(60,9,32)$ & $B_4$ & $a g v_3^2$ \\
$(66,10,35)$ & $\tau B_5$ & $\tau h_2 n g v_3^2$ \\
$(90,12,48)$ & $M^2$ & $d_0 g v_3^4 + h_1^6 v_3^6$ \\
\end{longtable}

Sometimes, analysis of a particular Adams differential requires
knowledge of the algebraic structure of $\Ext_\C$ in
much higher dimensions.  If the higher dimension is not too large,
then one can rely on explicit machine computations.
But if the higher dimension goes beyond the current range
of machine computations,
then results such as Corollary \ref{cor:main} can be of
great use.
See \cite{IWX} for specific examples of precisely this situation.

The subalgebra $A(2)$ of the $\C$-motivic Steenrod algebra
is of particular importance because there exists a 
$\C$-motivic modular forms spectrum $\mmf$ \cite{GIKR18}
whose cohomology is isomorphic to $A//A(2)$.
This implies that $\Ext_{A(2)}$ is the $E_2$-page of the
Adams spectral sequence that converges to the homotopy groups
of $\mmf$.

One might hope that the sub-Hopf algebra $B$ is similarly 
realizable.  We will show in Theorem \ref{thm:realize}
that it is not.
In other words, while
$\Ext_B$ is useful for studying the 
algebraic structure of the $\C$-motivic Adams $E_2$-page
$\Ext_\C$, it cannot be used to study Adams differentials.

\section{A subalgebra of the $\C$-motivic Steenrod algebra}

Recall that 
the dual $\C$-motivic Steenrod algebra $A_*$ \cite{Voevodsky03b}
\cite{Isaksen14} takes the form
\[
\frac{\M_2[\tau_0, \tau_1, \ldots, \xi_1, \xi_2, \ldots]}
{\tau_i^2 = \tau \xi_{i+1}},
\]
where $\M_2 = \F_2[\tau]$.
The coproduct of $A_*$ is given by the formulas
\[
\tau_i \mapsto \tau_i \otimes 1 +
\sum_{k = 0}^i \xi_{i-k}^{2^k} \otimes \tau_{k}
\quad \quad \quad \quad \quad
\xi_i \mapsto 
\sum_{k = 0}^i \xi_{i-k}^{2^k} \otimes \xi_{k}.
\]
By convention, we let $\xi_0$ equal $1$.

\begin{defn}
\label{defn:B}
Let $B_*$ be the quotient 
\[
\frac{\M_2[\tau_0, \tau_1, \tau_2, \tau_3, \xi_1, \xi_2]}
{\tau_0^2 + \tau \xi_1, \xi_1^4, \tau_1^2 + \tau \xi_2, \xi_2^2, 
\tau_2^2, \tau_3^2}
\]
of $A_*$.
Let $B$ be the dual subobject of $A$.
\end{defn}

\begin{remark}
The classical dual Steenrod algebra $A^\cl_*$ takes the form
$\F_2[\zeta_1, \zeta_2, \ldots]$, which is
the result of inverting $\tau$
in $A_*$, where $\tau_i$ and $\xi_{i+1}$ correspond to
$\zeta_{i+1}$ and $\zeta_{i+1}^2$ respectively.
The classical analogue of $B_*$ is the 
quotient
\[
\frac{\F_2[\zeta_1, \zeta_2, \zeta_3, \zeta_4]}
{\zeta_1^8, \zeta_2^4, \zeta_3^2, \zeta_4^2}.
\]
\end{remark}

\begin{lemma}
\label{lem:B-split}
The quotient $B_*$ is a Hopf algebra that splits as
\[
\frac{\M_2[\tau_0, \tau_1, \tau_2, \xi_1, \xi_2]}
{\tau_0^2 + \tau \xi_1, \xi_1^4, \tau_1^2 + \tau \xi_2, \xi_2^2, \tau_2^2}
\underset{\M_2}{\otimes}
\frac{\M_2[\tau_3]}{\tau_3^2}.
\]
\end{lemma}

\begin{proof}
To check that $B_*$ is a Hopf algebra, one must verify that 
the coproduct is well-defined.  In other words, if $x$
is an element of $A_*$ that maps to zero in $B_*$, then
the coproduct of $x$ in $A_*$ also maps to zero in $B_* \otimes B_*$.
This follows from direct computation.

The splitting also follows from direct computation.  Namely,
the coproduct of $\tau_3$ in $A_*$ maps to 
$1 \otimes \tau_3 + \tau_3 \otimes 1$ in $B_* \otimes B_*$.
\end{proof}

\begin{prop}
\label{prop:Ext-B}
$\Ext_B = \Ext_B(\M_2, \M_2)$ is isomorphic to
$\M_2[v_3] \underset{\M_2}{\otimes} \Ext_{A(2)}$,
where $v_3$ has degree $(14,1,7)$.
\end{prop}

\begin{proof}
This follows immediately from the splitting of Lemma \ref{lem:B-split},
together with the observation that the cohomology of an exterior
algebra is a polynomial algebra.
\end{proof}

The projection map $p: A_* \map B_*$ induces a map
$p_*: \Ext_\C \map \Ext_B$.  We will use the map $p_*$
to detect some structural phenomena in $\Ext_\C$.

\section{Massey products in $\Ext_B$}

The map $p_*$ is induced by a map 
$\tilde{p}: C^*(A) \map C^*(B)$ of cobar complexes.
Note that $\tilde{p}$ is a 
map of differential graded algebras.
In particular, $C^*(B)$ is a right $C^*(A)$-module,
and therefore $\Ext_B$ is a right $\Ext_\C$-module.
By definition, $p_* (x)$ equals $1 \cdot x$, where
$1$ is the identity element of $\Ext_B$.

Moreover, the map $\tilde{p}$ makes
$\Ext_B$ into a ``Toda module" over
$\Ext_\C$, in the following sense.
For all $x$ in $\Ext_B$ and all $a$ and $b$ in $\Ext_{\C}$ such that
$x \cdot a$ and $a b$ are both zero, there is a bracket
$\langle x, a, b \rangle$ in $\Ext_B$.  These brackets satisfy the usual
properties.  Later in the proof of Proposition \ref{prop:exist}, 
we will use the
shuffling relation
\[
\langle x, a, b \rangle \cdot c = x \cdot \langle a, b, c \rangle,
\]
for $x$ in $\Ext_B$ and $a$, $b$, and $c$ in $\Ext_\C$.

For later use, we compute one particular bracket.
In $\Ext_\C$, there is an element
$g_2$ of degree $(44,4,24)$ that is detected by
$b_{22}^2$ in the May spectral sequence.
This element satisfies the relation $h_0^3 g_2 = 0$.

\begin{prop}
\label{prop:bracket}
The bracket $\langle 1, g_2, h_0^3 \rangle$ in $\Ext_B$
in degree $(45,6,24)$ equals
$e_0 v_3^2 + h_1^3 v_3^3$, with no indeterminacy.
\end{prop}

\begin{proof}
First, we should verify that the bracket is well-defined.
We need that 
$1 \cdot g_2$ equals zero in $\Ext_B$ in degree
$(44,4,24)$.  But $\Ext_B$ is zero in that degree,
so $1 \cdot g_2$ must be zero.

Next, we compute the indeterminacy.
By inspection, the only possible indeterminacy 
is generated by $1 \cdot h_0 h_5 d_0$.
But this expression is zero because
$1 \cdot h_5$ is zero in $\Ext_B$ for degree reasons.

The map $\tilde{p}$ induces a map
of May spectral sequences \cite{May64} \cite{Isaksen14}.
The May $E_1$-page that converges to
$\Ext_\C$ has generators of the form $h_{ij}$ with $i \geq 1$
and $j \geq 0$.
On the other hand, the May $E_1$-page that converges to
$\Ext_B$ has generators
$h_0$, $h_1$, $h_2$, $h_{20}$,
$h_{21}$, $h_{30}$, and $h_{40}$.
The map of May spectral sequences takes $h_{ij}$ to the element
of the same name, or to zero if the element is not present
in the May $E_1$-page for $\Ext_B$.

We will compute the bracket $\langle 1, g_2, h_0^3 \rangle$
in $\Ext_B$ using the May Convergence Theorem 
\cite{May69} \cite{Isaksen14}.  Beware that this theorem
has a technical hypothesis involving the behavior of higher
``crossing" differentials.  In our specific case,
this technical hypothesis is satisfied for degree reasons.

The key point is that there is a 
May differential 
\[
d_6( (b_{21} b_{40} + b_{30} b_{31}) h_0(1) ) = h_0^3 g_2.
\]
Therefore, $\langle 1, g_2, h_0^3 \rangle$
is detected by the image of  
$(b_{21} b_{40} + b_{30} b_{31}) h_0(1)$
in the May spectral sequence for $\Ext_B$.
By inspection, this image equals
$h_{40}^2 b_{21} h_0(1)$. 

Finally, we must determine the elements of $\Ext_B$
that are detected by $h_{40}^2 b_{21} h_0(1)$ in the
May spectral sequence.
Note that
$e_0$ is detected by $b_{21} h_0(1)$ and
$v_3$ is detected by $h_{40}$.
However, beware that the element $h_1^3 v_3^3$ is detected
by $h_1^3 h_{40}^3$ in lower May filtration.
Consequently,
$h_{40}^2 b_{21} h_0(1)$ detects both
$e_0 v_3^2$ and $e_0 v_3^2 + h_1^3 v_3^3$.

We have now shown that $\langle 1, g_2, h_0^3 \rangle$
equals either
$e_0 v_3^2$ or $e_0 v_3^2 + h_1^3 v_3^3$.
Finally, we must distinguish between these two cases.

Recall from \cite{GI15}*{p.\ 4729} that there is a relation
$M h_1^6 = e_0^3 + d_0 \cdot e_0 g$ in $\Ext_\C$.
Apply $p_*$ to obtain a relation in $\Ext_B$.
We have that
\[
p_*(M h_1^6) = 1 \cdot \langle g_2, h_0^3, h_1^6 \rangle =
\langle 1, g_2, h_0^3 \rangle h_1^6.
\]
Here, we are using the well-known Massey product
\[
M h_1^6 = \langle g_2, h_0^3, h_1 \rangle h_1^5
\]
(see, for example, \cite{Isaksen14}*{Table 16}).
So the possible values for $p_*(M h_1^6)$
are $h_1^6 e_0 v_3^2$ and
$h_1^6 e_0 v_3^2 + h_1^9 v_3^3$.

The possible values for $p_*(e_0)$ are $0$, $e_0$, $h_1^3 v_3$,
and $e_0 + h_1^3 v_3$, so the possible values of
$p_*(e_0^3)$ are $0$, $e_0^3$, $h_1^9 v_3^3$, and
$e_0^3 + h_1^3 e_0^2 v_3 + h_1^6 e_0 v_3^2 + h_1^9 v_3^3$.

The only possible value for $p_*(d_0)$ is $d_0$.
The possible values for $p_*(e_0 g)$ are
$0$, $e_0 g$, $h_1^3 g v_3$, and $e_0 g + h_1^3 g v_3$.
(Recall that $e_0 g$ is an indecomposable element in
$\Ext_\C$.)
Therefore, the possible values for $p_*(d_0 \cdot e_0 g)$
are
$0$, $e_0^3$, $h_1^3 e_0^2 v_3$, and $e_0^3 + h_1^3 e_0^2 v_3$.
Here, we are using the relation $e_0^2 = d_0 g$ in $\Ext_{A(2)}$.

By inspection, 
the only consistent possibilities are that
$p_*(e_0) = e_0 + h_1^3 v_3$,
$p_*(e_0 g) = e_0 g + h_1^3 g v_3$,
and $p_*(M h_1^6) = h_1^6 e_0 v_3^2 + h_1^9 v_3^3$.
\end{proof}

\begin{remark}
The May spectral sequence argument in the proof of 
Proposition \ref{prop:bracket} is much the same as the
corresponding proof in \cite{MPT70}.
However,
the complications involving $h_1^3 v_3^3$ are new.
\end{remark}

\begin{remark}
The careful reader may wonder about the definitional distinction
between $e_0$ and $e_0 + h_1^3 v_3$.  Cannot $e_0$ in $\Ext_B$
be defined to be the value of $p_*(e_0)$?  The answer lies in the
multiplicative structure of $\Ext_{A(2)}$.  There is a relation
$h_1^2 e_0 = c_0 u$ in $\Ext_{A(2)}$.
From the formula $p_*(e_0) = e_0 + h_1^3 v_3$, it follows that
$p_*(h_1^2 e_0)$ is not divisible by $c_0$ or $u$ in $\Ext_B$.
This multiplicative fact is not consistent with the possibility
that $p_*(e_0) = e_0$ under a different choice of basis.
\end{remark}

\section{The Mahowald operator}

\begin{defn}
Let $x$ be an element of $\Ext_\C$ such that $h_0^3 x = 0$.
Define $M x$ to be the Massey product 
$\langle g_2, h_0^3, x \rangle$.
\end{defn}

As always, the Massey product $M x$ can have indeterminacy.
In other words, $M x$ may be a set of elements, not just a single
well-defined element.

If $h_0^3 x = 0$, then 
the iterated Massey products $M^k x = M (M^{k-1} x)$
are defined for all $k \geq 1$.  This follows from the computation
that
\[
h_0^3 \langle g_2, h_0^3, x \rangle =
\langle h_0^3, g_2, h_0^3 \rangle x = 0
\]
because $\langle h_0^3, g_2, h_0^3 \rangle = 0$.

\begin{prop}
\label{prop:exist}
Let $x$ be an element of $\Ext_\C$ such that $h_0^3 x = 0$.
Then 
$p_*(M x)$ equals $(e_0 v_3^2 + h_1^3 v_3^3) p_*(x)$
in $\Ext_B$.
\end{prop}

In particular, Proposition \ref{prop:exist} implies that
$p_*(M x)$ always consists of a single element, even if
$M x$ has indeterminacy.

\begin{proof}
Consider the shuffling relation
\[
p_*(M x) = 
1 \cdot \langle g_2, h_0^3, x \rangle =
\langle 1, g_2, h_0^3 \rangle \cdot x.
\]
Proposition \ref{prop:bracket} computes the second
bracket.
\end{proof}

\begin{remark}
We have stated our results in terms of Massey products
of the form $\langle g_2, h_0^3, x \rangle$.  However, they also apply
to Massey products of the form $\langle h_0 g_2, h_0^2, x \rangle$
and $\langle h_0^2 g_2, h_0, x \rangle$, using the shuffling
relations
\[
\langle h_0^2 g_2, h_0, x \rangle \subseteq
\langle h_0 g_2, h_0^2, x \rangle \subseteq
\langle g_2, h_0^3, x \rangle.
\]
\end{remark}

\section{$h_1$-periodic $\Ext$}

In the spirit of \cite{GI15}, one ought to study the 
$h_1$-periodic maps
\[
\xymatrix@1{
\Ext_\C [h_1^{-1}] \ar[r]^{p_*} & \Ext_B [h_1^{-1}] 
\ar[r] & \Ext_{A(2)}[h_1^{-1}].
}
\]
Computationally, this diagram is
\[
\xymatrix@1{
\F_2[h_1^{\pm 1}] [v_1^4, v_2, v_3, \ldots] \ar[r]^{p_*} &
\F_2[h_1^{\pm 1}] [v_1^4, v_2, v_3, u] \ar[r] &
\F_2[h_1^{\pm 1}] [v_1^4, v_2, u].
}
\]
Both maps take $v_1^4$ to $v_1^4$.
Moreover, the composition takes
$v_n$ to $v_2 u^{2^{n-2} - 1}$ 
\cite{GI15}*{Conjecture 5.5 and Proposition 6.4} \cite{GIKR18}.

For degree reasons, $p_*$ takes $v_2$ to $v_2$.
The computations of $p_*(e_0)$ and $p_*(e_0 g)$ at the end of the proof 
of Proposition \ref{prop:bracket} imply that
$p_*(v_3) = v_3 + v_2 u$ and that
$p_*(v_4) = v_3 u^2 + v_2 u^3$.
We suspect that $p_*(v_n) = v_3 u^{2^{n-2} - 2} + 
v_2 u^{2^{n-2} - 1}$ in general, although we have not actually
computed this formula.

On the other hand, the map
\[
\Ext_B [h_1^{-1}] \map \Ext_{A(2)}[h_1^{-1}]
\]
takes $v_2$ and $u$ to the elements of the same name 
in the target, and it must take $v_3$ to $0$.

\section{Non-Realizability}

The purpose of Theorem \ref{thm:realize} is that
$\Ext_B$ is not the $E_2$-page of an Adams $E_2$-page.
In other words, while
$\Ext_B$ is useful for studying the 
algebraic structure of the $\C$-motivic Adams $E_2$-page
$\Ext_\C$, it cannot be used to study Adams differentials.

\begin{thm}
\label{thm:realize}
There does not exist a $\C$-motivic ring spectrum $E$ equipped with
ring map $f: E \map \mmf$
such that the $\F_2$-motivic cohomology of $E$ is $A//B$,
and such that $f$ induces the projection
\[
A//A(2) \map A//B
\]
in cohomology.
\end{thm}

\begin{proof}
Suppose that $E$ exists.  
The unit map $S^{0,0} \map E$
induces a map of Adams spectral sequences.  On $E_2$-pages,
this map is
$p_*:\Ext_\C \map \Ext_B$.
By Theorem \ref{thm:main},
the element $M h_1$ of $\Ext_\C$ maps to 
$h_1 e_0 v_3^3 + h_1^4 v_3^3$ in $\Ext_B$.
Since $M h_1$ is a permanent cycle in the Adams spectral sequence
for the $\C$-motivic sphere spectrum
\cite{Isaksen14},
it follows by naturality that
$h_1 e_0 v_3^3 + h_1^4 v_3^3$ is a permanent cycle in
the Adams spectral sequence for $E$.

On the other hand,
the map $f$ also induces a map of Adams spectral sequences.  On
$E_2$-pages, this map takes the form
$\Ext_B \map \Ext_{A(2)}$.
The element $v_3$ must be a permanent cycle for degree reasons.
Also, $d_2(e_0) = h_1^2 d_0$ in the Adams spectral sequence
for $\mmf$.  By naturality of $f$, it follows that
$d_2(h_1 e_0 v_3^3 + h_1^4 v_3^3) = h_1^3 d_0 v_3^3$.

This contradiction shows that $E$ cannot exist.
\end{proof}

\begin{remark}
\label{rem:B-realize}
One can also pose an analogous question about a classical spectrum
whose cohomology is
$A^\cl//B^\cl$.
Such a classical spectrum also does not exist, for essentially
the same reasons.  However, one must use the non-zero 
classical differential
$d_3(e_0) = P c_0$ in the Adams spectral sequence for $\tmf$.
\end{remark}

\bibliographystyle{amsalpha}

\end{document}